\documentclass[11pt, a4paper, fleqn, final]{article}
\usepackage[utf8]{inputenc}
\usepackage[T1]{fontenc}
\usepackage[english]{babel}
\usepackage{amsmath, amsthm, amssymb, mathtools, dsfont}
\usepackage{csquotes}
\usepackage[backend=bibtex,sortcites=true,isbn=false,maxnames=5,mincrossrefs=1,parentracker=true,sorting=nyt]{biblatex}
\usepackage{hyperref, url}
\usepackage{graphicx}
\usepackage{xcolor}
\usepackage{fancyhdr}
\usepackage{pdfpages}
\usepackage{libertine}

\newtheorem{theo}{Theorem}
\newtheorem{coro}[theo]{Corollary}
\newtheorem{prop}[theo]{Proposition}
\newtheorem{lemm}[theo]{Lemma}
\newtheorem{conj}[theo]{Conjecture}

\theoremstyle{definition}
\newtheorem{defi}[theo]{Definition}
\newtheorem{rema}[theo]{Remark}

\newcommand{\eqspace}{\ensuremath{\mathrel{\phantom{=}}}}
\DeclarePairedDelimiter\abs{\lvert}{\rvert}
\DeclareMathOperator{\sgn}{sgn}

\makeatletter
\def\moverlay{\mathpalette\mov@rlay}
\def\mov@rlay#1#2{\leavevmode\vtop{%
   \baselineskip\z@skip \lineskiplimit-\maxdimen
   \ialign{\hfil$\m@th#1##$\hfil\cr#2\crcr}}}
\newcommand{\charfusion}[3][\mathord]{
    #1{\ifx#1\mathop\vphantom{#2}\fi
        \mathpalette\mov@rlay{#2\cr#3}
      }
    \ifx#1\mathop\expandafter\displaylimits\fi}
\makeatother
\newcommand{\cdotcup}{\charfusion[\mathbin]{\cup}{\cdot}}

\title{The Merrifield-Simmons conjecture \\ also holds for parity graphs}
\author{Martin Trinks\thanks{Supported by NSFC grant No. 11371205.} \\
\small Center for Combinatorics \\[-0.8ex]
\small Nankai University \\[-0.8ex]
\small Tianjin, China \\
\small \tt martin.trinks@googlemail.com
}
\date{}

\begin{document}

\maketitle

\begin{abstract}
The Merrifield-Simmons conjectures states a relation between the distance of vertices in a simple graph $G$ and the number of independent sets, denoted as $\sigma(G)$, in vertex-deleted subgraphs. Namely, that the sign of the term $\sigma(G_{-u}) \cdot \sigma(G_{-v}) - \sigma(G) \cdot \sigma(G_{-u-v})$ only depends on the parity of the distance of $u$ and $v$ in $G$. We prove this statement in the case of parity graphs and give some evidence that this result may not be further generalized to other classes of graphs.
\end{abstract}

\section{Introduction}

Let $G = (V, E)$ be a simple graph and $\sigma(G)$ the number of independent (vertex) sets of $G$, that is the number of vertex subsets $W \subseteq V$ such that no two vertices of $W$ are adjacent \cite{merrifield1980, merrifield1989}. In chemistry this number is also known as Merrifield-Simmons index. For two vertices $u, v \in V$, the term $\Delta(G, u, v)$ is defined as
\begin{align}
\Delta(G, u, v) = \sigma(G_{-u}) \cdot \sigma(G_{-v}) - \sigma(G) \cdot \sigma(G_{-u-v}),
\end{align}
where $G_{-w}$ is the graph with the vertex $w$ and its incident edges removed.
The \emph{Merrifield-Simmons conjecture} (MSC) states that $\sgn(\Delta(G, u, v))$, the sign of $\Delta(G, u, v)$, only depends on the distance between the vertices $u$ and $v$ in $G$, denoted by $d(G, u, v)$.

\begin{conj}[Merrifield-Simmons conjecture]
\label{conj:msc}
Let $G = (V, E)$ be a simple (bipartite) graph and $u, v \in V$ two vertices. Then
\begin{align}
\sgn(\Delta(G, u, v)) &= (-1)^{d(G, u, v) + 1}. \label{eq:conj_msc}
\end{align}
\end{conj}

Merrifield and Simmons \cite[page 144]{merrifield1989} noted the statement above as a property (without proof), but did not mention the class of graphs they were considering. \textcite{gutman1990} mentioned some counterexamples for arbitrary simple graphs and explicitly restated the conjecture for bipartite graphs. He also confirmed the statement for trees \cite{gutman1991}. The present author proved the MSC in the case of bipartite graphs \cite{trinks2013}. For more previous results see \cite{gutman1990, gutman1991, li1996, trinks2011, trinks2013, wang2001}.

This paper aims to show in which graphs classes the conjecture holds. To prove the Merrifield-Simmons conjecture (MSC) for parity graphs we go along nearly the same line of arguments as in the bipartite case, but in some clarified and generalized version. Thus, in Section \ref{sec:gmsc} we introduce generalizations of the terms used in the MSC to vertex subsets and some properties of them, on which the main theorem given in Section \ref{sec:parity_graphs} is based. In Section \ref{sec:counterexamples} we conclude by presenting counterexamples which give some evidence that the result cannot be further generalized. In the reminder of this section we provide the necessary notation for graphs and the applied properties for the number of independent sets.

For a simple graph $G = (V, E)$ with a vertex $v \in V$ and a vertex subset $W \subseteq V$ we use the following notations: $G_{-W}$ denotes the graph $G$ where all vertices $v \in W$ are deleted, that is these vertices and their incident edges are removed. The open neighborhood of $W$ is denoted by $N_{G}(W)$, that is the set of all vertices adjacent to a vertex $v \in W$. If $W = \{v\}$ then we write $G_{-v}$ and $N_{G}(v)$ instead of $G_{-\{v\}}$ and $N_{G}(\{v\})$, respectively. $G_1 \cdotcup G_2$ is the disjoint union of the graphs $G_1$ and $G_2$, that is the union of disjoint copies of both graphs. %For all other notation we refer to \cite{diestel2005}.

For the number of independent sets $\sigma(G)$ we use the following basic properties: First, it is multiplicative in components, that is 
\begin{align}
\sigma(G_1 \cdotcup G_2) = \sigma(G_1) \cdot \sigma(G_2). \label{eq:sigma_prop_4}
\end{align}
Second, it satisfies for each vertex $v \in V$ the recurrence relation
\begin{align}
\sigma(G) &= \sigma(G_{-v}) + \sigma(G_{-v-N_G(v)}). \label{eq:sigma_prop_1}
\end{align}
Finally, this recurrence relation can be generalized to vertex subsets:
\begin{theo}[Theorem 3.7 in \cite{hoede1994}]
\label{theo:sigma_prop_5}
Let $G = (V, E)$ be a simple graph and $U \subseteq V$ a vertex subset. Then
\begin{align}
& \sigma(G) = \eqspace \sum_{\mathclap{\substack{W \subseteq U \\ W \text{ is independent}}}}{\sigma(G_{-U-N_{G}(W)})}. \label{eq:theo_sigma_prop_5}
\end{align}
\end{theo}
\section{A generalization for vertex subsets}
\label{sec:gmsc}

In the following, a generalization of $\Delta(G, u, v)$ is considered where vertex subsets instead of vertices are deleted.

\begin{defi}
Let $G = (V, E)$ be a simple graph and $A, B \subseteq V$ two vertex subsets. Then $\Delta(G, A, B)$ is defined as 
\begin{align}
\Delta(G, A, B) = \sigma(G_{-A}) \cdot \sigma(G_{-B}) - \sigma(G) \cdot \sigma(G_{-A-B}).
\end{align}
\end{defi}

This generalization has the advantage that a recurrence relation for $\Delta(G, A, B)$ can be derived which enables us to state the term for $G$ as a sum over terms for proper subgraphs of $G$. In fact, in the case of bipartite graphs \cite{trinks2013} this recurrence relation (and Proposition \ref{prop:delta_non-disjoint} as well) are ``hidden'' in the proof, here we state them explicitly.

\begin{lemm}
\label{lemm:delta_neighborhood}
Let $G = (V, E)$ be a simple graph and $A, B \subseteq V$ two disjoint vertex subsets. Then
\begin{align}
\Delta(G, A, B) &= - \sum_{\mathclap{\substack{W \subseteq A \\ W \text{ is independent}}}}{\Delta(G_{-A}, N_{G}(W), B)}.
\end{align}
\end{lemm}

\begin{proof}
Applying the recurrence relation for vertex subsets (Theorem \ref{theo:sigma_prop_5}) we obtain
\begin{align*}
\Delta(G, A, B) 
&= \sigma(G_{-A}) \cdot \sigma(G_{-B}) - \sigma(G) \cdot \sigma(G_{-A-B}) \\
&= \sigma(G_{-A}) \cdot \sum_{\mathclap{\text{ind. } W \subseteq A }}{\sigma(G_{-B-A-N_{G_{-B}}(W)})} - \sum_{\mathclap{\text{ind. } W \subseteq A}}{\sigma(G_{-A-N_{G}(W)})} \cdot \sigma(G_{-A-B}) \\
&= \phantom{-} \sum_{\mathclap{\text{ind. } W \subseteq A}}{\left[ \sigma(G_{-A}) \cdot \sigma(G_{-B-A-N_{G_{-B}}(W)}) - \sigma(G_{-A-N_{G}(W)}) \cdot \sigma(G_{-A-B}) \right]}.
\end{align*}
As $A$ and $B$ are disjoint, for all $W \subseteq A$ we have $B \cup N_{G_{-B}}(W) = B \cup N_{G}(W)$ and consequently $G_{-B-A-N_{G_{-B}}(W)} = G_{-B-A-N_{G}(W)}$. Applying this, the statement follows:
\begin{align*}
\Delta(G, A, B)
&= \phantom{-} \sum_{\mathclap{\text{ind. } W \subseteq A}}{\left[\sigma(G_{-A}) \cdot \sigma(G_{-B-A-N_{G}(W)}) - \sigma(G_{-A-N_{G}(W)}) \cdot \sigma(G_{-A-B})\right]} \\
&= - \sum_{\mathclap{\text{ind. } W \subseteq A}}{\left[ \sigma(G_{-A-N_{G}(W)}) \cdot \sigma(G_{-A-B}) - \sigma(G_{-A}) \cdot \sigma(G_{-A-N_{G}(W)-B}) \right]} \\
&= - \sum_{\mathclap{\text{ind. } W \subseteq A}}{\Delta(G_{-A}, N_{G}(W), B)}.  \qedhere
\end{align*}
\end{proof}

Let $G^A$, $G^B$, $G^{AB}$ and $G^*$ denote the union of those connected components of $GH$ including vertices from $A$, from $B$, from $A$ and $B$, and from neither of both, respectively. If there are no connected components which include vertices from both vertex subsets $A$ and $B$, that means $G = G^A \cdotcup G^B \cdotcup G^*$ and $G^{AB} = \emptyset$, then the terms in $\Delta(G, A, B)$ cancel each other.

\begin{prop}[Corollay 5 in \cite{trinks2013}] \label{prop:delta_infinity}
Let $G = (V, E)$ be a simple graph and $A, B \subseteq V$ two vertex subsets, such that $G = G^A \cdotcup G^B \cdotcup G^*$. Then
\begin{align}
\Delta(G, A, B) = 0.
\end{align}
\end{prop}

\begin{proof}
The vertices of $A$ and $B$ can only be deleted in $G^A$ and $G^B$, respectively. Thus, the statement follows via
\begin{align*}
\Delta(G, A, B) 
&= \sigma(G_{-A}) \cdot \sigma(G_{-B}) - \sigma(G) \cdot \sigma(G_{-A-B}) \\
&= \sigma((G^A \cdotcup G^B \cdotcup G^*)_{-A}) \cdot \sigma((G^A \cdotcup G^B \cdotcup G^*)_{-B}) \\
& \eqspace - \sigma(G^A \cdotcup G^B \cdotcup G^*) \cdot \sigma((G^A \cdotcup G^B \cdotcup G^*)_{-A-B}) \\
&= \sigma(G^A_{-A} \cdotcup G^B \cdotcup G^*) \cdot \sigma(G^A \cdotcup G^B_{-B} \cdotcup G^*) \\
& \eqspace - \sigma(G^A \cdotcup G^B \cdotcup G^*) \cdot \sigma(G^A_{-A} \cdotcup G^B_{-B} \cdotcup G^*) \\
&= \sigma(G^A_{-A}) \cdot \sigma(G^B) \cdot \sigma(G^*) \cdot \sigma(G^A) \cdot \sigma(G^B_{-B}) \cdot \sigma(G^*) \\
& \eqspace - \sigma(G^A) \cdot \sigma(G^B) \cdot \sigma(G^*) \cdot \sigma(G^A_{-A}) \cdot \sigma(G^B_{-B}) \cdot \sigma(G^*) \\
&= 0. \qedhere
\end{align*}
\end{proof}

\begin{prop}
\label{prop:delta_non-disjoint}
Let $G = (V, E)$ be a simple graph and $A, B \subseteq V$ two vertex subsets, such that $A \cap B = C \neq \emptyset$. Then
\begin{align}
\Delta(G, A, B) < \Delta(G_{-C}, A \setminus C, B \setminus C).
\end{align}
\end{prop}

\begin{proof}
The statement follows by applying the recurrence relation for vertex subsets (Theorem \ref{theo:sigma_prop_5}):
\begin{align*}
\Delta(G, A, B) &= \sigma(G_{-A}) \cdot \sigma(G_{-B}) - \sigma(G) \cdot \sigma(G_{-A-B}) \\
&= \sigma(G_{-A}) \cdot \sigma(G_{-B}) - \sum_{\mathclap{\substack{W \subseteq C \\ W \text{ is independent}}}}{\sigma(G_{-C-N_{G}(W)})} \cdot \sigma(G_{-A-B}) \\
&= \sigma(G_{-A}) \cdot \sigma(G_{-B}) - \sigma(G_{-C}) \cdot \sigma(G_{-A-B}) \\
& \eqspace - \sum_{\mathclap{\substack{\emptyset \subset W \subseteq C \\ W \text{ is independent}}}}{\sigma(G_{-C-N_{G}(W)})} \cdot \sigma(G_{-A-B}) \\
& < \sigma(G_{-A}) \cdot \sigma(G_{-B}) - \sigma(G_{-C}) \cdot \sigma(G_{-A-B}) \\
&= \sigma(G_{-C-(A \setminus C)}) \cdot \sigma(G_{-C-(B \setminus C)}) \\
& \eqspace - \sigma(G_{-C}) \cdot \sigma(G_{-C-(A \setminus C) - (B \setminus C)}) \\
&= \Delta(G_{-C}, A \setminus C, B \setminus C). \qedhere
\end{align*}
\end{proof}

In order to generalize the notion of distance between a pair of vertices to distance between two vertex subsets, the set of chord-free paths connecting vertices of the two vertex subsets are considered.

\begin{defi}
\label{defi:induced_A_B_path}
Let $G = (V, E)$ be a graph and $A, B \subseteq V$ two vertex subsets. A path $P = (v_1, \ldots, v_k)$ of $G$ is an \emph{induced $A$-$B$-path}, if $V(P) \cap A = \{v_1\}$ and $V(P) \cap B = \{v_k\}$, where $V(P)$ is the set of vertices of $P$, and $\{v_i, v_j\} \in E \Leftrightarrow \abs{i - j} = 1$. By $P_i(G, A, B)$ we denote the \emph{set of all induced $A$-$B$-paths} in $G$. The \emph{length} of an induced $A$-$B$-path $P$ is the number of edges in $P$, that means $|V(P)| - 1$. 
\end{defi}

\begin{defi}
Let $G = (V, E)$ be a graph, $A, B \subseteq V$ two disjoint vertex subsets. We say $P_i(G, A, B)$ is \emph{even} (\emph{odd}) if the length of each path $P \in P_i(G, A, B)$ is even (odd) and $P_i(G, A, B)$ is \emph{infinite}, if there is no induced $A$-$B$-path in $G$ (the length of each $P \in P_i(G, A, B)$ is infinite).
\end{defi}

\begin{lemm}[Lemma 6 in \cite{trinks2013}]
\label{lemm:gmsc_prop}
Let $G = (V, E)$ be a graph, $A, B \subseteq V$ two disjoint vertex subsets and $W \subseteq A$ a subset of $A$. If $P_i(G, A, B)$ is even (odd), then $P_i(G_{-A}, N_{G}(W), B)$ is odd (even) or infinite. There is at least one vertex subset $W \subseteq A$, such that $P_i(G_{-A}, N_{G}(W), B)$ is not infinite and hence odd (even), namely $W = \{a\}$ where $a \in A$ is connected by an induced $A$-$B$-path  in $G$ to a vertex $b \in B$.
\end{lemm}

\begin{proof}
The first part is shown by contradiction. Assume $P_i(G, A, B)$ is even (odd) and for a subset $W \subseteq A$ there is an even (odd) induced $N_{G}(W)$-$B$-path in $G_{-A}$, connecting a vertex $x \in N_{G}(W)$ with a vertex $b \in B$. Because $x \in N_{G}(W)$, there is a vertex $a \in W \subseteq A$, such that $a$ and $x$ are adjacent. As $x$ is the only vertex of the path in $N_{G}(W)$ by definition, $a$ is non-adjacent to all other of its vertices. Hence, the path from $a$ to $x$ to $b$ in $G$ is induced and has odd (even) length, which contradicts the assumption of the statement.

As $P_i(G, A, B)$ is even (odd), there is at least one induced $A$-$B$-path $P$ in $G$. Thus, there is a vertex $a \in A$ connected by an induced $A$-$B$-path to a vertex $b \in B$. Consequently, there is a induced $N_{G}(a)$-$B$-path in $G_{-A}$, which proves the second part.
\end{proof}
\section{MSC for parity graphs}
\label{sec:parity_graphs}

\begin{theo}
\label{theo:gmsc_parity}
Let $G = (V, E)$ be a simple graph and $A, B \subseteq V$ two vertex subsets. Then
\begin{align}
\Delta(G, A, B) &= \sigma(G_{-A}) \cdot \sigma(G_{-B}) - \sigma(G) \cdot \sigma(G_{-A-B}) \notag \\
& \begin{cases}
< 0 & \text{ if } P_i(G, A, B) \text{ is even,} \\
= 0 & \text{ if } P_i(G, A, B) \text{ is infinite,} \\
> 0 & \text{ if } P_i(G, A, B) \text{ is odd.}
\end{cases}
\end{align}
\end{theo}

\begin{proof}
If $P_i(G, A, B)$ is infinite, then there are no connected components including vertices from both vertex subsets $A$ and $B$. Thus, this case is stated in Proposition \ref{prop:delta_infinity}. Therefore, from now on we assume that $P_i(G, A, B)$ is not infinite, that means there is at least one vertex $a \in A$ and at least one vertex $b \in B$ connected by a path. 

We prove the two cases $P_i(G, A, B)$ is even and $P_i(G, A, B)$ is odd by induction with respect to the number of vertices in $G$, denoted by $n(G)$. 

For the basic step we assume a graph $G$ with the minimal number of vertices, this is $n(G) = 1$ if $P_i(G, A, B)$ is even and $n(G) = 2$ if $P_i(G, A, B)$ is odd. For $P_i(G, A, B)$ is even and $n(G) = 1$ we have $G = (\{a\}, \emptyset)$ and $A = B = \{a\}$. Hence
\begin{align*}
\Delta(G, A, B) &= \sigma(G_{-A}) \cdot \sigma(G_{-B}) - \sigma(G) \cdot \sigma(G_{-A-B}) = -1 < 0.
\end{align*}
For $P_i(G, A, B)$ is odd and $n(G) =  2$ we have $G = (\{a,b\}, \{\{a,b\}\})$ and $A = \{a\}$, $B = \{b\}$. Hence
\begin{align*}
\Delta(G, A, B) &= \sigma(G_{-A}) \cdot \sigma(G_{-B}) - \sigma(G) \cdot \sigma(G_{-A-B}) = 1 > 0.
\end{align*}

We assume as induction hypothesis that the statement holds for any graph with at most $k$ vertices and consider from now on a graph $G$ with $n(G) = k+1$ vertices. 

If $A$ and $B$ are not disjoint, that means $A \cap B = C \neq \emptyset$, which means that $P_i(G, A, B)$ is even, then by Proposition \ref{prop:delta_non-disjoint} we have
\begin{align*}
\Delta(G, A, B) < \Delta(G_{-C}, A \setminus C, B \setminus C).
\end{align*}
As $C$ is non-empty, $G_{-C}$ has at most $k$ vertices and hence we can use the induction hypothesis. Furthermore, as $P_i(G, A, B)$ is even, $P_i(G_{-C}, A \setminus C, B \setminus C)$ is also even or infinite (by deleting $C$, no new paths occur, but some are destroyed), that means
\begin{align*}
\Delta(G, A, B) < \Delta(G_{-C}, A \setminus C, B \setminus C) \leq 0.
\end{align*}

Otherwise, if $A$ and $B$ are disjoint, we can apply Lemma \ref{lemm:delta_neighborhood}:
\begin{align*}
\Delta(G, A, B) 
&= - \sum_{\mathclap{\text{ind. } W \subseteq A}}{\Delta(G_{-A}, N_{G}(W), B)}.
\end{align*}
$A$ is non-empty (otherwise $P(G, A, B)$ would be infinite), therefore $G_{-A}$ has at most $k$ vertices and the induction hypothesis can be applied: For all $W \subseteq A$ we have
\begin{align*}
\Delta(G_{-A}, N_{G}(W), B)
& \begin{cases}
\geq 0 & \text{ if } P_i(G, A, B) \text{ is even,} \\
\leq 0 & \text{ if } P_i(G, A, B) \text{ is odd,}
\end{cases}
\end{align*}
because if $P_i(G, A, B)$ is even (odd), then $P_i(G_{-A}, N_{G}(W), B)$ is not even (odd) by Lemma \ref{lemm:gmsc_prop}. But at least for $W = \{a\} \subseteq A$ we have
\begin{align*}
\Delta(G_{-A}, N_{G}(W), B)
& \begin{cases}
> 0 & \text{ if } P(G, A, B) \text{ is even,} \\
< 0 & \text{ if } P(G, A, B) \text{ is odd,}
\end{cases}
\end{align*}
again by Lemma \ref{lemm:gmsc_prop}. Hence, we get the other two cases of the statement:
\begin{align*}
\Delta(G, A, B)
& \begin{cases}
< 0 & \text{ if } P(G, A, B) \text{ is even,} \\
> 0 & \text{ if } P(G, A, B) \text{ is odd.} 
\end{cases} \qedhere
\end{align*}
\end{proof}

\begin{defi}
A simple graph $G = (V, E)$ is a \emph{parity graph}, if for any two vertices $u, v \in V$ the length of all induced $u$-$v$-paths in $G$ has the same parity.
\end{defi}

Parity graphs are a generalization of bipartite graphs, because only the length of all induced $u$-$v$-path is claimed to have the same parity, instead of all $u$-$v$-path as for bipartite graphs.

If two vertices have even (odd) distance in a parity graph, then all induced paths have even (odd) length and hence the previous theorem proves the MSC (Conjecture \ref{conj:msc}) for parity graphs (and arbitrary vertices).
\begin{coro}
\label{coro:msc_parity_graph}
The Merrifield-Simmons conjecture holds for parity graphs.
\end{coro}

In relation to the corollary above, Theorem \ref{theo:gmsc_parity} is slightly more general, because there, only assumptions about the subgraph connecting the vertex subsets are made: The MSC holds in a graph $G = (V, E)$ for vertex subsets $A, B \subseteq V$, if the subgraph induced by all vertices in some $A$-$B$-path is a parity graph.
\section{Counterexamples}
\label{sec:counterexamples}

Having shown in the preceding section that the Merrifield-Simmons conjecture (MSC) not only holds in bipartite graphs, but also holds in parity graphs, the question arises if it can be further generalized to larger graph classes.

It seems that this is not possible, because of the graphs displayed in Figure \ref{fig:counterexample_msc}, where $G_1$ is the minimal counterexample for the MSC conjecture in arbitrary graphs.

\begin{figure}
\begin{center}
	\includegraphics{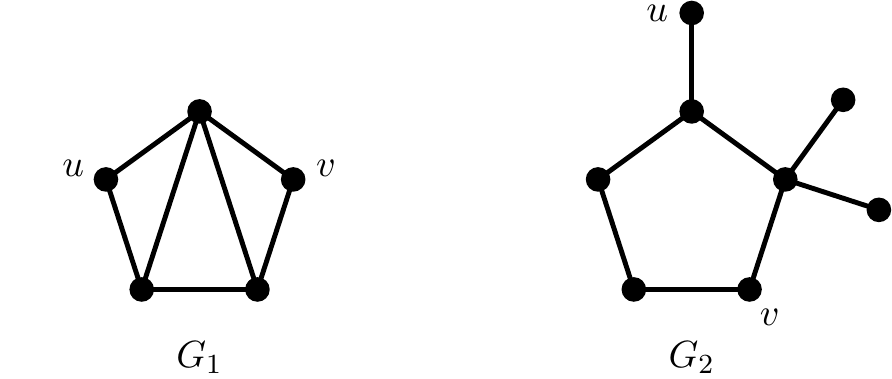}
\end{center}
\caption{Graphs $G_1$ and $G_2$, which are counterexamples for the MSC in superclasses of parity graphs. It holds $\Delta(G_1, u, v) = 6 \cdot 6 - 9 \cdot 4 = 0$ and $\Delta(G_2, u, v) = 21 \cdot 21 - 29 \cdot 15 = 6$.}
\label{fig:counterexample_msc}
\end{figure}

According to \textcite{isgci}, the following are the minimal superclasses of parity graphs: (5,2)-odd-chordal (equivalent to Meyniel, (odd building,odd-hole)-free, and very strongly perfect), $P_4$-bipartite, ($X_{38}$,gem,house)-free, preperfect, and skeletal.

\begin{rema}
The graphs $G_1$ and $G_2$ in Figure \ref{fig:counterexample_msc} provide counterexamples for the MSC conjecture. $G_1$ is a (5,2)-odd-chordal, $P_4$-bipartite, preperfect and skeletal graph, and $G_2$ is a ($X_{38}$,gem,house)-free graph. Consequently, the MSC cannot be generalized to any of the minimal superclasses of parity graphs listed by \textcite{isgci}.
\end{rema}

\printbibliography

\end{document}